\documentclass[11pt]{amsart}
\usepackage[a4paper,margin=1.1in]{geometry}
\usepackage{amsmath,amssymb,amsthm}
\usepackage[colorlinks=true,linkcolor=blue,citecolor=blue]{hyperref}

\newtheorem{theorem}{Theorem}
\newtheorem{lemma}[theorem]{Lemma}
\newtheorem{corollary}[theorem]{Corollary}
\newtheorem{proposition}[theorem]{Proposition}
\theoremstyle{definition}
\newtheorem{definition}[theorem]{Definition}
\theoremstyle{remark}
\newtheorem{remark}[theorem]{Remark}
\newtheorem{example}[theorem]{Example}
\theoremstyle{plain}
\newtheorem*{motzkin}{Motzkin's transposition theorem}

\newcommand{\R}{\mathbb{R}}
\newcommand{\Sn}{\mathcal{S}_n}
\newcommand{\Snp}{\mathcal{S}_n^{+}}
\newcommand{\On}{\mathcal{O}_n}
\newcommand{\V}{\mathcal{V}}
\newcommand{\Rn}{\mathcal{R}_n}
\newcommand{\Kn}{\mathcal{K}_n}
\newcommand{\En}{\mathcal{E}_n}
\newcommand{\YA}{\mathcal{Y}_A}
\newcommand{\Ln}{\mathcal{L}_n}

\newcommand{\Zc}{\mathcal{Z}}
\newcommand{\Hc}{\mathcal{H}_n}
\newcommand{\tr}{\operatorname{tr}}
\newcommand{\diag}{\operatorname{diag}}
\newcommand{\dg}{\operatorname{dg}}
\newcommand{\ip}[2]{\langle #1,#2\rangle}
\newcommand{\vd}{\mathbf{d}}
\newcommand{\ve}{\mathbf{e}}
\newcommand{\vq}{\mathbf{q}}
\newcommand{\vu}{\mathbf{u}}
\newcommand{\vv}{\mathbf{v}}
\newcommand{\vw}{\mathbf{w}}
\newcommand{\vx}{\mathbf{x}}
\newcommand{\vh}{\mathbf{h}}
\newcommand{\vg}{\mathbf{g}}
\newcommand{\va}{\mathbf{a}}
\newcommand{\vb}{\mathbf{b}}
\newcommand{\vc}{\mathbf{c}}
\newcommand{\vk}{\mathbf{k}}
\newcommand{\vm}{\mathbf{m}}
\newcommand{\vy}{\mathbf{y}}
\newcommand{\doi}[1]{\href{https://doi.org/#1}{\nolinkurl{doi:#1}}}

\begin{document}

\title[Guo's perturbation for symmetric nonnegative matrices]{A note on Guo's perturbation result\\ for symmetric nonnegative matrices}
\author{Helena \v{S}migoc}
\address{School of Mathematics and Statistics, University College Dublin, Belfield, Dublin 4, Ireland}
 \email{helena.smigoc@ucd.ie}
\date{Draft of 8 October 2026}

\begin{abstract}
In 1997 Guo Wuwen proved that if $(\lambda_1,\lambda_2,\dots,\lambda_n)$ is the spectrum of a
nonnegative matrix with Perron eigenvalue $\lambda_1$ and $\lambda_2$ real, then
$(\lambda_1+t,\lambda_2\pm t,\lambda_3,\dots,\lambda_n)$ is the spectrum of a nonnegative matrix for
every $t\ge 0$. Whether the same holds for symmetric nonnegative matrices has remained an open
question. We give an affirmative answer, in a slightly stronger form: if $A$ is a symmetric
nonnegative matrix with Perron eigenvalue $\lambda_1$ and further eigenvalues $\lambda_2,\dots,\lambda_n$, and
$d_1\ge|d_2|+\dots+|d_n|$, then there is a symmetric matrix $B\ge A$ with eigenvalues
$\lambda_1+d_1,\dots,\lambda_n+d_n$, and $B$ has the same diagonal as $A$ when
$d_1+\dots+d_n=0$. 

This leads us to call a symmetrically realizable list \emph{Guo-extreme} if none of Guo's moves can be reversed for this list. 
Every symmetrically realizable list is a
Guo-extreme list plus a vector from the cone of Guo's moves. Our second main result is a certificate,
modeled on a theorem of Laffey on extreme symmetric matrices: if every
nonnegative symmetric matrix that commutes with $A$ and is equal to zero on the entries where $A$ is positive is a multiple
of the identity, then $A$ is orthogonally similar to a positive matrix if $A$ has a positive
diagonal entry, and to a matrix with zero diagonal and positive off-diagonal entries if $A$ has zero
diagonal. Consequently, the spectrum of $A$ can be moved by a small amount in every direction,
respectively in every trace-preserving direction, and remain symmetrically realizable. In
particular, the spectrum of such a matrix is not Guo-extreme.
\end{abstract}

\maketitle

\section{Introduction and main results}

The nonnegative inverse eigenvalue problem (NIEP) asks which lists of complex numbers are spectra of
entrywise nonnegative matrices, and its symmetric version (SNIEP) asks which lists of real numbers are
spectra of symmetric nonnegative matrices. Both problems are open for lists of length $n\ge 5$. We refer
to the surveys \cite{ELN,JMPP}  for background, and to \cite{LM,LSp,Lo21,MP,Sp} for the
SNIEP with $n=5$, where the set of symmetrically realizable lists is still not completely known.

Results that produce new realizable lists from known ones are useful tools in both problems. Guo
\cite{G97} proved that if $(\lambda_1,\lambda_2,\dots,\lambda_n)$ is the spectrum of a nonnegative
matrix with Perron eigenvalue $\lambda_1$ and $\lambda_2$ is real, then
$(\lambda_1+t,\lambda_2\pm t,\lambda_3,\dots,\lambda_n)$ is the spectrum of a nonnegative matrix for
every $t\ge 0$, and Laffey \cite{L05} extended this result to non-real eigenvalues. On the symmetric side,
Fiedler \cite{F74} showed how to combine two symmetrically realizable lists, and further constructions of this type are given in
\cite{LS,SRMB}. Whether Guo's result holds for symmetric realizability is stated as an open question
in \cite{LS}. For $n\le 4$ the answer is positive for a simple reason: a
real list of length $n\le 4$ is symmetrically realizable if and only if $\lambda_1\ge|\lambda_i|$ for all
$i$ and $\sum_i\lambda_i\ge 0$, and these conditions are preserved by Guo's perturbation. Our first main result,
Theorem~\ref{thm:main} below shows that the result holds for all $n$.

\subsection{Notation}
Throughout, $\Sn$ is the space of real symmetric $n\times n$ matrices with the inner product
$\ip{X}{Y}=\tr(XY)$, $\Snp\subset\Sn$ is the set of (entrywise) nonnegative symmetric matrices of order
$n$, $\On$ is the set of orthogonal matrices of order $n$, $\Ln$ is the space of real skew-symmetric
matrices of order $n$, $J$ is the all-ones matrix, and inequalities
between matrices or vectors are entrywise.
The standard
basis of $\R^n$ is denoted by $\ve_1,\dots,\ve_n$. For a vector $\vx\in\R^n$ we write $x_i$ for its $i$-th
entry, so that $\vx=(x_1,\dots,x_n)$. This also applies to expressions such as $(A\vx)_i$. For lists of
eigenvalues we usually write $\sigma=(\lambda_1,\dots,\lambda_n)$ instead. For $\vv\in\R^n$, $\diag(\vv)$
is the diagonal matrix with diagonal entries $v_1,\dots,v_n$. Conversely, for $X\in\R^{n\times n}$,
$\dg(X)=(x_{11},\dots,x_{nn})$ is the vector of its diagonal entries, so that $\dg(\diag(\vv))=\vv$. For matrices $X=(x_{pq})$ and
$Y=(y_{pq})$ of the same size, $X\circ Y=(x_{pq}y_{pq})$ denotes their Hadamard (entrywise) product.

We write $\Rn$ for the set of all $\sigma=(\lambda_1,\dots,\lambda_n)\in\R^n$ such that
$\lambda_1,\dots,\lambda_n$ are the eigenvalues of some $A\in\Snp$ and $\lambda_1=\rho(A)$, no order is
imposed on $\lambda_2,\dots,\lambda_n$. Thus $\Rn$ is the set of symmetrically realizable lists of
length $n$, written with the Perron eigenvalue in the first position. The set $\Rn$ is closed. 
Moreover, since $\rho(A)\ge|\lambda_i|$ and $\tr A^k\ge 0$, every
$\sigma\in\Rn$ satisfies
\begin{equation}\label{eq:nec}
\lambda_1\ge|\lambda_i|\ \ (i=2,\dots,n) \text{ and } \sum_{i=1}^n\lambda_i^k \ge 0
\text{ for all } k \in \mathbb{N}.
\end{equation}
Further necessary conditions for $\sigma\in\Rn$ are known in
the literature. See, for example, the surveys \cite{ELN,JMPP}.

Let
\begin{equation}\label{eq:cone}
\Kn=\Bigl\{\vd=(d_1,\dots,d_n)\in\R^n:\ d_1\ge\sum_{i=2}^n|d_i|\Bigr\}.
\end{equation}

For $n\ge 2$, $\Kn$ is a closed convex cone, $d_1>0$ for every $\vd\in\Kn\setminus\{0\}$, and every
$\vd\in\Kn$ is a nonnegative combination of the $2(n-1)$ vectors $\ve_1\pm \ve_i$, $i=2,\dots,n$: if
$\epsilon_i\in\{1,-1\}$ is such that $\epsilon_id_i=|d_i|$, then
\begin{equation}\label{eq:Kngen}
\vd=\sum_{i=2}^n|d_i|\,(\ve_1+\epsilon_i\ve_i)+\Bigl(d_1-\sum_{i=2}^n|d_i|\Bigr)\ve_1 ,
\qquad \ve_1=\tfrac12(\ve_1+\ve_2)+\tfrac12(\ve_1-\ve_2).
\end{equation}
We also write
\[
\Hc=\Bigl\{\vd\in\R^n:\ \sum_{i=1}^nd_i=0\Bigr\}
\]
for the hyperplane of trace-preserving directions.

\subsection{Overview of Main Results} 
The main theorem of this paper is stated below. 

\begin{theorem}\label{thm:main}
Let $A\in\Snp$ have eigenvalues $\lambda_1,\lambda_2,\dots,\lambda_n$ with $\lambda_1=\rho(A)$, and let
$\vd=(d_1,\dots,d_n)\in\Kn$. Then there is a matrix $B\in\Snp$ with eigenvalues $\lambda_1+d_1,\lambda_2+d_2,\dots,\lambda_n+d_n$
such that $B\ge A$ entrywise. 
\end{theorem}

If $\vd\in\Hc$ in Theorem~\ref{thm:main}, then $\tr B=\tr A$, and since $b_{ii}\ge a_{ii}$ for all
$i$, the matrix $B$ has the same diagonal as $A$. The symmetric version of Guo's result is an immediate consequence of Theorem \ref{thm:main}.

\begin{corollary}[symmetric version of Guo's result]\label{cor:guo}
Let $(\lambda_1,\dots,\lambda_n)$ be the spectrum of a matrix $A\in\Snp$ with Perron
eigenvalue $\lambda_1$, let $j\in\{2,\dots,n\}$ and $t\ge 0$. Then both lists
\[
(\lambda_1+t,\dots,\lambda_j-t,\dots,\lambda_n)\qquad\text{and}\qquad
(\lambda_1+t,\dots,\lambda_j+t,\dots,\lambda_n)
\]
are spectra of matrices $B\in\Snp$ with $B\ge A$. In the first case $B$ has the same diagonal as $A$.
\end{corollary}

\begin{proof}
Apply Theorem~\ref{thm:main} with $\vd=t(\ve_1\mp \ve_j)$. In the first case $\vd\in\Hc$, so $B$ has the same diagonal as $A$.
\end{proof}

In terms of the sets introduced above, Theorem~\ref{thm:main} implies that
\begin{equation}\label{eq:RnKn}
\Rn+\Kn\subseteq\Rn .
\end{equation}
Indeed, let $\sigma\in\Rn$ and $\vd\in\Kn$. Theorem~\ref{thm:main}, applied to a matrix realizing
$\sigma$, gives $B\in\Snp$ with eigenvalues $\lambda_i+d_i$. Since $\lambda_1\ge\lambda_i$ and $d_1\ge d_i$ for
all $i$, the largest of them is $\lambda_1+d_1$, so $\lambda_1+d_1=\rho(B)$ and $\sigma+\vd\in\Rn$.

Section~\ref{sec:prelim} collects the tools used in the proofs: a theorem of the alternative
derived from Motzkin's transposition theorem and a form of the submersion theorem. Section~\ref{sec:proof} is dedicated to the proof of Theorem~\ref{thm:main}, which is not
constructive.

In Section~\ref{sec:extreme} we call a list $\sigma\in\Rn$ \emph{Guo-extreme} if none of the moves in
Corollary~\ref{cor:guo} is reversible at $\sigma$, that is, if $\sigma-t(\ve_1\pm \ve_i)\notin\Rn$ for all
$i\ge 2$ and $t>0$. Every list in $\Rn$ is a Guo-extreme list plus a vector from $\Kn$
(Proposition~\ref{prop:decomposition}), so the SNIEP reduces to the description of the Guo-extreme
lists. In Section~\ref{sec:certificate} we give sufficient conditions for a list not to be
Guo-extreme. We characterize the spectra of positive symmetric matrices and of symmetric nonnegative
matrices with positive off-diagonal entries in terms of the cone $\Kn$
(Propositions~\ref{prop:perron} and~\ref{prop:offdiag}). Our second main result,
Theorem~\ref{thm:certificate}, is a certificate, modeled on a theorem of Laffey \cite[Theorem~4]{L98},
which guarantees that a given matrix $A\in\Snp$ is orthogonally similar to a positive matrix or to a
matrix with zero diagonal and positive off-diagonal entries. Consequently, the spectrum of $A$ is not
Guo-extreme (Corollary~\ref{cor:moves}). Section~\ref{sec:certificate} ends with a comparison of
the notions of extremality that appear in the paper.

\section{Preliminary lemmas}\label{sec:prelim}

This section collects two classical tools, Motzkin's transposition theorem and a form of the
submersion theorem, together with their matrix forms, which are used in the proof.
Lemma~\ref{lem:alternative} is a theorem of the alternative, a direct consequence of Motzkin's
transposition theorem \cite{Mo36}. We recall the theorem in the form given in \cite[Theorem~2]{BI}
(see also \cite[Chapter~7]{S}). Lemma~\ref{lem:curve} is a form of the submersion theorem, and
Lemma~\ref{lem:correction} is its consequence for irreducible matrices with zero diagonal.

In the proofs of the main results, Lemma~\ref{lem:alternative} controls the first-order
term of a perturbation (in Lemma~\ref{lem:tangent} and Theorem~\ref{thm:certificate}), while
Lemmas~\ref{lem:curve} and~\ref{lem:correction} deal with higher-order terms: they correct a curve at
second order so that it satisfies an exact constraint (in Theorem~\ref{thm:certificate} and
Corollary~\ref{cor:moves}).

\begin{motzkin}
Let $G_{\scriptscriptstyle=}$, $G_{\scriptscriptstyle\ge}$ and $G_{\scriptscriptstyle>}$ be real matrices with the same number of columns, where $G_{\scriptscriptstyle>}$
has at least one row ($G_{\scriptscriptstyle=}$ and $G_{\scriptscriptstyle\ge}$ may have no rows, in which case the corresponding conditions and terms are absent). Then exactly one of the following holds:
\begin{enumerate}
\item[(i)] there is a real vector $\vh$ with $G_{\scriptscriptstyle=}\vh=0$, $G_{\scriptscriptstyle\ge}\vh\ge0$ and $G_{\scriptscriptstyle>}\vh>0$;
\item[(ii)] there are real vectors $\va$, $\vb\ge0$ and $\vc\ge0$ with $\vc\ne0$ such that
$\va^{\top} G_{\scriptscriptstyle=}=\vb^{\top} G_{\scriptscriptstyle\ge}+\vc^{\top} G_{\scriptscriptstyle>}$.
\end{enumerate}
\end{motzkin}

\noindent The two alternatives exclude each other, since a solution $\vh$ of (i) and $\va,\vb,\vc$ as in
(ii) would give $0=\va^{\top}G_{\scriptscriptstyle=}\vh=\vb^{\top}G_{\scriptscriptstyle\ge}\vh+\vc^{\top}G_{\scriptscriptstyle>}\vh>0$.

Motzkin's theorem concerns homogeneous systems. The systems in the proof of
Lemma~\ref{lem:alternative} contain constant terms, and we use the following inhomogeneous version.

\begin{corollary}\label{cor:motzkin}
Let $G_{\scriptscriptstyle=}$, $G_{\scriptscriptstyle\ge}$ and $G_{\scriptscriptstyle>}$ be real matrices with the same number of
columns, and let $\vg_{\scriptscriptstyle=}$, $\vg_{\scriptscriptstyle\ge}$ and $\vg_{\scriptscriptstyle>}$ be real vectors
whose numbers of entries equal the numbers of rows of $G_{\scriptscriptstyle=}$, $G_{\scriptscriptstyle\ge}$ and
$G_{\scriptscriptstyle>}$, respectively. Any of the three matrices may have no rows, in which case the
corresponding vector is empty. Then exactly one of the following holds:
\begin{enumerate}
\item[(i)] there is a real vector $\vh$ with $G_{\scriptscriptstyle=}\vh+\vg_{\scriptscriptstyle=}=0$,
$G_{\scriptscriptstyle\ge}\vh+\vg_{\scriptscriptstyle\ge}\ge0$ and $G_{\scriptscriptstyle>}\vh+\vg_{\scriptscriptstyle>}>0$;
\item[(ii)] there are real vectors $\va$, $\vb\ge0$ and $\vc\ge0$ and a number $c_0\ge0$ with
$(\vc,c_0)\ne0$ such that
\[
\va^{\top}G_{\scriptscriptstyle=}=\vb^{\top}G_{\scriptscriptstyle\ge}+\vc^{\top}G_{\scriptscriptstyle>}
\qquad\text{and}\qquad
\va^{\top}\vg_{\scriptscriptstyle=}=\vb^{\top}\vg_{\scriptscriptstyle\ge}+\vc^{\top}\vg_{\scriptscriptstyle>}+c_0 .
\]
\end{enumerate}
\end{corollary}

\begin{proof}
We introduce an additional unknown $\tau\in\R$ that multiplies the constant terms, and consider the
homogeneous system
\begin{equation}\label{eq:homog}
G_{\scriptscriptstyle=}\vh+\tau\vg_{\scriptscriptstyle=}=0,\qquad G_{\scriptscriptstyle\ge}\vh+\tau\vg_{\scriptscriptstyle\ge}\ge0,\qquad
G_{\scriptscriptstyle>}\vh+\tau\vg_{\scriptscriptstyle>}>0,\qquad \tau>0
\end{equation}
in $(\vh,\tau)$. This system has a solution if and only if the system in (i) has one. If $\vh$ is  a solution of (i), then
 $(\vh,1)$ is a solution of \eqref{eq:homog}. Conversely, if $(\vh,\tau)$ solves
\eqref{eq:homog}, then $\vh/\tau$ solves (i). The strict inequality $\tau>0$ provides a row of strict
inequalities even when $G_{\scriptscriptstyle>}$ has no rows (without it, $(\vh,\tau)=(0,0)$ would always be a solution).  This homogenization is a standard method in
linear programming.

Motzkin's theorem, applied to \eqref{eq:homog} with the matrices
\[
\bigl(G_{\scriptscriptstyle=}\ \ \vg_{\scriptscriptstyle=}\bigr),\qquad
\bigl(G_{\scriptscriptstyle\ge}\ \ \vg_{\scriptscriptstyle\ge}\bigr)\qquad\text{and}\qquad
\begin{pmatrix}G_{\scriptscriptstyle>}&\vg_{\scriptscriptstyle>}\\0&1\end{pmatrix},
\]
shows that \eqref{eq:homog} has no solution if and only if there are $\va$, $\vb\ge0$ and
$(\vc,c_0)\ge0$ with $(\vc,c_0)\ne0$ and
$\va^{\top}\bigl(G_{\scriptscriptstyle=}\ \ \vg_{\scriptscriptstyle=}\bigr)=\vb^{\top}\bigl(G_{\scriptscriptstyle\ge}\ \ \vg_{\scriptscriptstyle\ge}\bigr)
+\vc^{\top}\bigl(G_{\scriptscriptstyle>}\ \ \vg_{\scriptscriptstyle>}\bigr)+c_0\,(0\ \ 1)$.
Comparing the columns gives (ii). 
\end{proof}

We want to apply Corollary~\ref{cor:motzkin} to systems whose unknowns and constraints are matrices rather
than vectors. To bring such systems into the vector form of the corollary, we introduce the following
notation. We index the entries of a symmetric matrix by the set
\[
\Pi=\{(p,q):\ 1\le p\le q\le n\}
\]
of $\binom{n+1}{2}$ positions. Let  $\Delta=\{(p,p):\ 1\le p\le n\}$ be the set of
diagonal positions, and $\Pi_{\mathcal{L}}=\Pi\setminus \Delta$. For $W\in\Sn$ let $\operatorname{svec}(W)
\in\R^{\Pi}$ be the vector with entries $w_{pp}$ at the positions $(p,p)$ and $\sqrt2\,w_{pq}$ at the
positions $(p,q)$, $p<q$ (its weighted triangular part), and for
$K\in\Ln$ let $\operatorname{vec}_\wedge(K)\in\R^{\Pi_{\mathcal{L}}}$ be the vector of its
entries $k_{pq}$, $p<q$ (its strictly triangular part). In both cases the entries are listed in
lexicographic order of the index pairs $(p,q)$, that is, $(1,1),(1,2),\dots,(1,n),(2,2),\dots$ for
$\operatorname{svec}$ and $(1,2),(1,3),\dots,(1,n),(2,3),\dots$ for
$\operatorname{vec}_\wedge$. 
The weight $\sqrt2$ is introduced because each off-diagonal entry of a symmetric matrix occurs
twice in the trace inner product: $\ip{Y}{W}=\sum_pw_{pp}y_{pp}+2\sum_{p<q}w_{pq}y_{pq}$, so that
\[
\ip{Y}{W}=\operatorname{svec}(Y)^{\top}\operatorname{svec}(W)\qquad\text{for all }Y,W\in\Sn .
\]
Since the weights are positive, an entry of $\operatorname{svec}(W)$ is positive, zero or negative
exactly when the corresponding entry of $W$ is. In particular,  
 $\operatorname{svec}(W)\ge0$ if and only
if $W\ge0$.

For a set
$\mathcal{S}\subseteq\Pi$, a vector $\vv\in\R^\Pi$ and a matrix $L$ whose rows are indexed by $\Pi$, we
write $\vv[\mathcal{S}]$ for the subvector of $\vv$ with entries indexed by $\mathcal{S}$, and
$L[\mathcal{S}]$ for the submatrix of $L$ formed by the rows indexed by $\mathcal{S}$. (Since all
columns are always selected, we write $L[\mathcal{S}]$ rather than the more common
$L[\mathcal{S},:]$.)

For $B\in\Sn$
let
\[
\V_B=\{KB-BK:\ K\in\Ln\},
\]
a linear subspace of $\Sn$.

\begin{lemma}\label{lem:alternative}
Let $B,M\in\Sn$, and let $\Pi=\mathcal{P}\cup\mathcal{E}\cup\mathcal{F}$ be a partition of the set of
positions into three disjoint sets, some of which may be empty. Let
\begin{align*}
\mathcal{C}_\ge&=\{W\in\Sn:\ w_{pq}\ge 0\ \text{for }(p,q)\in\mathcal{P},\ w_{pq}=0\ \text{for }
(p,q)\in\mathcal{E}\},\\
\mathcal{C}_>&=\{W\in\Sn:\ w_{pq}>0\ \text{for }(p,q)\in\mathcal{P},\ w_{pq}=0\ \text{for }
(p,q)\in\mathcal{E}\},\\
\mathcal{D}&=\{Y\in\Sn:\ y_{pq}\ge 0\ \text{for }(p,q)\in\mathcal{P},\ y_{pq}=0\ \text{for }
(p,q)\in\mathcal{F}\},
\end{align*}
where the entries in the positions not mentioned are arbitrary.
\begin{enumerate}
\item[(a)] If $(M+\V_B)\cap\mathcal{C}_\ge=\emptyset$, 
then there is $Y\in\mathcal{D}$ such that $BY=YB$ and $\ip{Y}{M}<0$.
\item[(b)] If $(M+\V_B)\cap\mathcal{C}_>=\emptyset$, 
then there is $Y\in\mathcal{D}$ such that $BY=YB$ and $\ip{Y}{M}\le 0$. Moreover, $\ip{Y}{M}<0$ or $y_{pq}\ne 0$
for some $(p,q)\in\mathcal{P}$.
\end{enumerate}
\end{lemma}

Note that $\mathcal{D}$ is the dual cone of $\mathcal{C}_\ge$ (and of
$\mathcal{C}_>$): $Y\in\mathcal{D}$ if and only if $\ip{Y}{W}\ge0$ for all $W\in\mathcal{C}_\ge$.

\begin{proof}
We first note that, for $Y\in\Sn$,
\begin{equation}\label{eq:L3C1}
\ip{Y}{V}=0\text{ for all }V\in\V_B \iff BY=YB.
\end{equation}
Indeed, $\ip{Y}{KB-BK}=\tr\bigl(Y(KB-BK)\bigr)=\tr\bigl(K(BY-YB)\bigr)$, the matrix $S=BY-YB$ is
skew-symmetric, and taking $K=S^{\top}$ gives $\tr(S^{\top}S)=0$, so $S=0$. The converse is clear.

To prove (a) and (b), we write the conditions $X\in M+\V_B$ and $X\in\mathcal{C}_\ge$ (or $X\in\mathcal{C}_>$) as a linear system in the vector $\operatorname{vec}_\wedge(K)$, apply Corollary~\ref{cor:motzkin} to it, and translate the alternative system back into the matrix $Y$.

Let $L_B$ be the $|\Pi|\times |\Pi_{\mathcal L}|$ matrix with $\operatorname{svec}(KB-BK)=L_B\operatorname{vec}_\wedge(K)$
for all $K\in\Ln$, and put $\vm=\operatorname{svec}(M)$. Then $X\in M+\V_B$ if and only
if $\operatorname{svec}(X)=\vm+L_B\vk$ for some $\vk\in\R^{\Pi_{\mathcal{L}}}$. For $Y\in\Sn$ and $\vy=\operatorname{svec}(Y)$
we have $\vy^{\top}L_B\vk=\ip{Y}{KB-BK}$ with $\vk=\operatorname{vec}_\wedge(K)$. By \eqref{eq:L3C1}, $BY=YB$ holds if and only if
$\vy^{\top}L_B\vk=0$ for all $\vk\in\R^{\Pi_{\mathcal{L}}}$, that is, if and only if $\vy^{\top}L_B=0$. Hence
\begin{equation}\label{eq:L3C2}
\vy^{\top}L_B=0\iff BY=YB,\qquad\text{and}\qquad \vy^{\top}\vm=\ip{Y}{M}.
\end{equation}
(a) There is $X\in(M+\V_B)\cap\mathcal{C}_\ge$ if and only if there is $\vk\in\R^{\Pi_{\mathcal{L}}}$ with
\[
L_B[\mathcal{E}]\vk+\vm[\mathcal{E}]=0\qquad\text{and}\qquad L_B[\mathcal{P}]\vk+\vm[\mathcal{P}]\ge0 .
\]
This is alternative (i) of Corollary~\ref{cor:motzkin} with $\vh=\vk$,
$G_{\scriptscriptstyle=}=L_B[\mathcal{E}]$, $\vg_{\scriptscriptstyle=}=\vm[\mathcal{E}]$,
$G_{\scriptscriptstyle\ge}=L_B[\mathcal{P}]$, $\vg_{\scriptscriptstyle\ge}=\vm[\mathcal{P}]$, and
$G_{\scriptscriptstyle>}$ void. If it has no solution, alternative (ii) gives vectors $\va$, $\vb\ge0$
and a number $c_0>0$ (as $\vc$ is void) with
\[
\va^{\top}L_B[\mathcal{E}]=\vb^{\top}L_B[\mathcal{P}]\qquad\text{and}\qquad
\va^{\top}\vm[\mathcal{E}]=\vb^{\top}\vm[\mathcal{P}]+c_0 .
\]
Let $\vy\in\R^\Pi$ have
$\vy[\mathcal{P}]=\vb$, $\vy[\mathcal{E}]=-\va$ and $\vy[\mathcal{F}]=0$. Then 
$\vy^{\top}L_B=\vb^{\top}L_B[\mathcal{P}]-\va^{\top}L_B[\mathcal{E}]=0$ and
$\vy^{\top}\vm=\vb^{\top}\vm[\mathcal{P}]-\va^{\top}\vm[\mathcal{E}]=-c_0$, so by \eqref{eq:L3C2} the matrix $Y\in\Sn$ with $\operatorname{svec}(Y)=\vy$
lies in $\mathcal{D}$, commutes with $B$ and satisfies $\ip{Y}{M}<0$.

(b) Similarly, there is $X\in(M+\V_B)\cap\mathcal{C}_>$ if and only if there is $\vk\in\R^{\Pi_{\mathcal{L}}}$ with
\[
L_B[\mathcal{E}]\vk+\vm[\mathcal{E}]=0\qquad\text{and}\qquad L_B[\mathcal{P}]\vk+\vm[\mathcal{P}]>0 ,
\]
which is alternative (i) of Corollary~\ref{cor:motzkin} with $\vh=\vk$,
$G_{\scriptscriptstyle=}=L_B[\mathcal{E}]$, $\vg_{\scriptscriptstyle=}=\vm[\mathcal{E}]$,
$G_{\scriptscriptstyle\ge}$ void, $G_{\scriptscriptstyle>}=L_B[\mathcal{P}]$ and
$\vg_{\scriptscriptstyle>}=\vm[\mathcal{P}]$. If this system has no solution, alternative (ii) gives vectors
$\va$, $\vc\ge0$ and a number $c_0\ge0$, with $(\vc,c_0)\ne0$, such that
$\va^{\top}L_B[\mathcal{E}]=\vc^{\top}L_B[\mathcal{P}]$ and
$\va^{\top}\vm[\mathcal{E}]=\vc^{\top}\vm[\mathcal{P}]+c_0$. Let $\vy[\mathcal{P}]=\vc$,
$\vy[\mathcal{E}]=-\va$ and $\vy[\mathcal{F}]=0$. As in (a), the matrix $Y$ with
$\operatorname{svec}(Y)=\vy$ lies in $\mathcal{D}$, commutes with $B$, and $\ip{Y}{M}=-c_0\le0$. If
$\ip{Y}{M}=0$, then $c_0=0$, so $\vy[\mathcal{P}]=\vc\ne0$.
\end{proof}

In Lemma~\ref{lem:correction} we need a version of the implicit function theorem in which a partial derivative is
surjective but not necessarily invertible. It is a form of the submersion theorem: under the
assumptions of the following lemma the zero set of $F$ is, near $(0,0)$, a smooth submanifold whose
tangent space at $(0,0)$ is the kernel of the derivative of $F$; see for example
\cite[Chapter~5]{Lee}. For completeness we include a proof that uses only the usual implicit function theorem.

\begin{lemma}\label{lem:curve}
Let $\mathcal{X}$ and $\mathcal{Z}$ be finite-dimensional real vector spaces, and let $F$ be a smooth
map from a neighborhood of $(0,0)$ in $\mathcal{X}\times\R$ to $\mathcal{Z}$ with $F(0,0)=0$. Let
$T\colon\mathcal{X}\to\mathcal{Z}$ be the partial derivative of $F$ with respect to the first variable
at $(0,0)$, and let $v=\frac{\partial F}{\partial s}(0,0)\in\mathcal{Z}$ be its partial derivative with
respect to the second variable. Suppose that $T$ is surjective and that $\xi\in\mathcal{X}$ satisfies
\[
T\xi+v=0 .
\]
Then there is an $\varepsilon>0$ and a smooth curve $s\mapsto x(s)\in\mathcal{X}$, defined for $s$ with $|s|<\varepsilon$, such that
\[
x(0)=0,\qquad x'(0)=\xi,\qquad F(x(s),s)=0 .
\]
\end{lemma}

\begin{proof}
Write
$D_xF(x,s)$ for the partial derivative of $F$ with respect to the first variable, a linear map from
$\mathcal{X}$ to $\mathcal{Z}$. In particular, $D_xF(0,0)=T$. 
The implicit function theorem requires an invertible partial derivative, and $T$ is assumed to be surjective but in
general not injective. We therefore choose a subspace $\mathcal{S}\subseteq\mathcal{X}$ complementary
to the kernel of $T$, so that the restriction of $T$ to $\mathcal{S}$ is a bijection onto
$\mathcal{Z}$, and define
\[
\Psi(S,s)=F(s\xi+S,s)
\]
for $(S,s)$ in a neighborhood of $(0,0)$ in $\mathcal{S}\times\R$. The map $\Psi$ is smooth,
$\Psi(0,0)=0$, and its partial derivative with respect to $S$ at $(0,0)$ is the restriction of $T$ to
$\mathcal{S}$, which is invertible. By the implicit function theorem \cite[Chapter~9]{Ru} there is an $\varepsilon>0$ and a
smooth function $s\mapsto S(s)\in\mathcal{S}$, defined for $s$ with $|s|<\varepsilon$, with
\[
S(0)=0\qquad\text{and}\qquad \Psi(S(s),s)=0 .
\]
Put $x(s)=s\xi+S(s)$, so that $F(x(s),s)=\Psi(S(s),s)=0$ for $|s|<\varepsilon$.  Since $x'(s)=\xi+S'(s)$, the chain rule gives
\[
\frac{d}{ds}F(x(s),s)=D_xF(x(s),s)\bigl(x'(s)\bigr)+\frac{\partial F}{\partial s}(x(s),s).
\]
At $s=0$ we have $x(0)=0$, $D_xF(0,0)=T$ and $\frac{\partial F}{\partial s}(0,0)=v$. Since
$F(x(s),s)$ vanishes identically, its derivative at $s=0$ is zero, and therefore
\[
0=T\bigl(\xi+S'(0)\bigr)+v=T\bigl(S'(0)\bigr)+\bigl(T\xi+v\bigr)=T\bigl(S'(0)\bigr),
\]
where we used the assumption $T\xi+v=0$. Since $S'(0)\in\mathcal{S}$ and $T$ is injective on
$\mathcal{S}$, it follows that $S'(0)=0$. Hence $x(0)=S(0)=0$, $x'(0)=\xi+S'(0)=\xi$ and
$F(x(s),s)=0$, so the curve $x(s)$ has the required properties.
\end{proof}

In the curve $x(s)=s\xi+S(s)$ of this proof, the term $s\xi$ prescribes the first-order behavior, and
$S(s)$ is a correction of order $s^2$. Next result is an application Lemma~\ref{lem:curve} to irreducible matrices with zero diagonal. 

\begin{lemma}\label{lem:correction}
Let $A\in\Snp$ be irreducible with zero diagonal, let $Q\in\On$ with
$Q^{\top}AQ=\diag(\lambda_1,\dots,\lambda_n)$, let $\vd\in\Hc$, and put
$M=Q\diag(\vd)Q^{\top}$.
\begin{enumerate}
\item[(i)] The linear map $T\colon\Ln\to\Hc$, $T(L)=\dg(LA-AL)$, is surjective. In particular, there
exists $K\in\Ln$ with $\dg(M+KA-AK)=0$.
\item[(ii)] If $K\in\Ln$ satisfies $\dg(M+KA-AK)=0$, then there exists a smooth curve $s\mapsto
L(s)\in\Ln$, defined for small $|s|$, such that $B(s)=e^{L(s)}(A+sM)e^{-L(s)}$ is symmetric with
eigenvalues $\lambda_j+sd_j$, $j=1,\dots,n$, has zero diagonal, and satisfies
$B(s)=A+s\,(M+KA-AK)+O(s^2)$.
\end{enumerate}
\end{lemma}

\begin{proof}
(i) Let $\vw\in\Hc$ be orthogonal to the image of $T$.  We aim to show that $\vw=0$. Put $W=\diag(\vw)$, and observe
$\vw^{\top}\dg(X)=\tr(WX)$ for every matrix $X$. The assumption on $\vw$ now implies:
\[
0=\tr\bigl(W(LA-AL)\bigr)=\tr\bigl(L(AW-WA)\bigr)\qquad\text{for all }L\in\Ln .
\]
Inserting $L=(AW-WA)^{\top} \in \Ln$ into the above equation gives
$AW=WA$. In particular, $a_{pq}(w_q-w_p)=0$ for all $p,q$, that is, $w_p=w_q$ whenever $a_{pq}\ne 0$. Since $A$ is
irreducible, any two indices are joined by a path in the graph of $A$, so all entries of $\vw$ are
equal, say to $c$. As $\vw\in\Hc$, we get $nc=\sum_pw_p=0$, hence $\vw=0$. Thus the image of $T$ is
equal to $\Hc$. In particular, since $\tr M=\sum_jd_j=0$, there exists $K \in \Ln$ so that $T(K)=-\dg(M)$. 

(ii) Define
\[
\Phi\colon\Ln\times\R\to\Hc,\qquad \Phi(L,s)=\dg\bigl(e^{L}(A+sM)e^{-L}\bigr).
\]
This function is well defined because $\tr\bigl(e^L(A+sM)e^{-L}\bigr)=\tr A+s\tr M=0$. The map $\Phi$ is
smooth, $\Phi(0,0)=0$, its partial derivative with respect to $L$ at $(0,0)$ is $T$, which is
surjective by (i), and $\frac{\partial\Phi}{\partial s}(0,0)=\dg(M)$. Hence
\[
T(K)+\frac{\partial\Phi}{\partial s}(0,0)=\dg(KA-AK+M)=0 .
\]
By Lemma~\ref{lem:curve}, applied with $\mathcal{X}=\Ln$, $\mathcal{Z}=\Hc$, $F=\Phi$ and $\xi=K$,
there is a smooth curve $s\mapsto L(s)\in\Ln$, defined for small $|s|$, with $L(0)=0$, $L'(0)=K$ and
$\Phi(L(s),s)=0$. The matrix $B(s)=e^{L(s)}(A+sM)e^{-L(s)}$ is symmetric and it has  eigenvalues
$\lambda_j+sd_j$, since it is orthogonally similar to
$A+sM=Q\bigl(\diag(\lambda_1,\dots,\lambda_n)+s\diag(\vd)\bigr)Q^{\top}$. Furthermore, $\dg(B(s))=\Phi(L(s),s)=0$, $B(0)=A$, and $B'(0)=L'(0)A-AL'(0)+M=M+KA-AK$. As $B$ is smooth, Taylor's theorem gives
$B(s)=A+s\,(M+KA-AK)+O(s^2)$.
\end{proof}

Nonnegativity of $A$ is not used in Lemma~\ref{lem:correction}. We assume it for convenience, since the lemma is applied only in the nonnegative setting. The argument works for every symmetric matrix $A$ with zero diagonal whose graph is connected.

\section{Proof of Theorem~\ref{thm:main}}\label{sec:proof}

We first prove a first-order version of Theorem~\ref{thm:main}, and then outline how it leads to
the theorem.

\subsection*{The first-order lemma}
Let $B\in\Snp$ have
the orthogonal diagonalization $B=Q\Lambda_BQ^{\top}$, and let $\vd\in\R^n$ be the direction in which we want
to move the spectrum of $B$. For every $K\in\Ln$ the matrix $e^{sK}$ is orthogonal, so
\begin{equation}\label{eq:curve}
C_K(s)\;=\;e^{sK}\,Q\,\bigl(\Lambda_B+s\diag(\vd)\bigr)\,Q^{\top}e^{-sK},\qquad s\ge 0,
\end{equation}
is a curve of symmetric matrices that starts at $C_K(0)=B$ and whose eigenvalues at time $s$ are exactly
the diagonal entries of $\Lambda_B+s\diag(\vd)$. Expanding in $s$ gives
\begin{equation}\label{eq:expansion}
C_K(s)\;=\;B+s\,\bigl(Q\diag(\vd)Q^{\top}+KB-BK\bigr)+O(s^2).
\end{equation}
In the linear term, $Q\diag(\vd)Q^{\top}$ moves the eigenvalues while keeping the eigenvectors fixed, and it
is in general not nonnegative. The commutator $KB-BK$ comes from rotating the eigenvectors and has no
effect on the spectrum. Lemma~\ref{lem:tangent} states that, for irreducible $B$ and $\vd\in\Kn$, the
linear term in \eqref{eq:expansion} is entrywise nonnegative for at least one
$K\in\Ln$. This is the sense in which the lemma is a first-order version of
Theorem~\ref{thm:main}: along the curve $C_K$ for this $K$, an entry of $B$ can decrease only through
the $O(s^2)$ term. The lemma alone does not give $C_K(s)\ge 0$ for small $s>0$, because an entry that
vanishes both in $B$ and in the linear term is controlled by the second-order term. For directions
$\vd$ in the interior of $\Kn$, however, the linear term can be made entrywise positive, and this is
the form used in the proof of Theorem~\ref{thm:main}.

\begin{lemma}[first-order version]\label{lem:tangent}
Let $B\in\Snp$ be irreducible, and let $Q=[\vq_1\ \cdots\ \vq_n]\in\On$ be such that $Q^{\top}BQ$
is diagonal and $\vq_1>0$ is the Perron vector of $B$. Let $\vd\in\Kn$. Then there
is a matrix $K\in\Ln$ such that
\[
X\;:=\;Q\,\diag(\vd)\,Q^{\top}+KB-BK\;\ge\;0 .
\]
If $\vd$ lies in the interior of $\Kn$, that is, $d_1>\sum_{i=2}^n|d_i|$, then $K$ can be chosen so
that $X>0$ entrywise.
\end{lemma}

\begin{proof}
Write $M=Q\diag(\vd)Q^{\top}$. We must show that $M+\V_B$ contains a nonnegative matrix. If it does
not, Lemma~\ref{lem:alternative}(a), applied with $\mathcal{P}=\Pi$ and
$\mathcal{E}=\mathcal{F}=\emptyset$, gives a matrix $Y\in\Snp$ that commutes with $B$ and satisfies
$\ip{Y}{M}<0$. We show that this is impossible.

Since $B$ is irreducible, $\rho(B)$ is a simple eigenvalue.
From $B(Y\vq_1)=YB\vq_1=\rho(B)\,Y\vq_1$ we get
$Y\vq_1=\mu \vq_1$ for some $\mu$. Since $Y\ge 0$ has the positive eigenvector $\vq_1$, the eigenvalue
$\mu$ is equal to $\rho(Y)$.
As $Y$ is symmetric, its
eigenvalues lie in $[-\rho(Y),\rho(Y)]$, so $|\vq_i^{\top}Y\vq_i|\le\rho(Y)$ for all $i$. Therefore
\[
\ip{Y}{M}=\tr\Bigl(Y\Bigl(\sum_{i=1}^n d_i\,\vq_i\vq_i^{\top}\Bigr)\Bigr)=\sum_{i=1}^n d_i\,\vq_i^{\top}Y\vq_i\;\ge\;\Bigl(d_1-\sum_{i=2}^n|d_i|\Bigr)\rho(Y)\;\ge\;0,
\]
since $\vd\in \Kn$, contradicting $\ip{Y}{M}<0$.

Now let $d_1>\sum_{i\ge2}|d_i|$, and suppose that $M+\V_B$ does not contain any positive matrices.
Lemma~\ref{lem:alternative}(b), applied with $\mathcal{P}=\Pi$ and
$\mathcal{E}=\mathcal{F}=\emptyset$, gives $Y\in\Snp$ that commutes with $B$ and satisfies
$\ip{Y}{M}\le 0$, and either $\ip{Y}{M}<0$ or $Y\ne0$. In both cases $Y\ne0$. A nonzero symmetric
matrix has a nonzero eigenvalue, so $\rho(Y)>0$, and the computation above gives
$\ip{Y}{M}\ge\bigl(d_1-\sum_{i\ge2}|d_i|\bigr)\rho(Y)>0$, a contradiction.
\end{proof}

\begin{remark}
If $\vd\in\Hc$ then $\tr X=0$, so the matrix $X\ge 0$ in Lemma~\ref{lem:tangent} has zero diagonal:
the diagonal entries of $C_K(s)$ in \eqref{eq:curve} differ from those of $B$ only by $O(s^2)$, and the
off-diagonal entries do not decrease to first order. The proof of the lemma shows more generally that the set of $\vd\in\R^n$ for which
$Q\diag(\vd)Q^{\top}+\V_B$ contains a nonnegative matrix is the dual cone of
\[
\bigl\{(\vq_1^{\top}Y\vq_1,\dots,\vq_n^{\top}Y\vq_n):\ Y\in\Snp,\ YB=BY\bigr\}, 
\]
and Lemma~\ref{lem:tangent} states that this cone contains $\Kn$.
\end{remark}

\subsection*{Outline of the proof of Theorem~\ref{thm:main}}
As in the statement of Theorem~\ref{thm:main}, let $A\in\Snp$ have eigenvalues
$\lambda_1,\dots,\lambda_n$ with $\lambda_1=\rho(A)$, and let $\vd=(d_1,\dots,d_n)\in\Kn$. Since $\Kn$ is
invariant under permutations of the coordinates $2,\dots,n$, we may permute the indices
$2,\dots,n$ of the eigenvalues and of $\vd$ simultaneously, and so assume that
$\lambda_1\ge\lambda_2\ge\dots\ge\lambda_n$. Let $\Lambda=\diag(\lambda_1,\dots,\lambda_n)$ and
$D(t)=\Lambda+t\,\diag(\vd)$ for $t\ge 0$. The symmetric matrices with eigenvalues $\lambda_i+td_i$,
$i=1,\dots,n$, are exactly the matrices $QD(t)Q^{\top}$ with $Q\in\On$, and they form a compact set.

 Instead
of constructing a path of matrices through these sets, we follow a single number: in Step~1 we define
$h(t)$ as the largest value of $\min_{p,q}(QD(t)Q^{\top}-A)_{pq}$ over $Q\in\On$, and we call a matrix
$QD(t)Q^{\top}$ for which $Q$ attains this maximum a \emph{best matrix} at time $t$ (in Step~1 such
a $Q$ is called a maximizer at $t$). Thus $h(t)$ measures how
close the best matrix with spectrum $D(t)$ comes to dominating $A$ entrywise, and $h(t)\ge 0$ means
precisely that some matrix with spectrum $D(t)$ dominates $A$. Theorem~\ref{thm:main} is the statement
$h(1)\ge 0$, and $h(0)\ge 0$ holds because $A$ itself has spectrum $D(0)$.

In Step~1 we assume that $A$ is positive and that $\vd$ lies in the interior of $\Kn$. We show that $h$
cannot become negative, in three stages.
\begin{enumerate}
\item[(i)] Let $B$ be a best matrix at time $t$. If $B$ is irreducible, Lemma~\ref{lem:tangent} provides
a curve \eqref{eq:curve} that starts at $B$, has spectrum $D(t+s)$ at time $s$, and whose linear term
\eqref{eq:expansion} is entrywise positive.
\item[(ii)] Along this curve, the first-order gain dominates the second-order error for small $s$. Since
the best matrix at time $t+s$ is at least as good as the point of the curve at time $s$, this gives
the \emph{local claim}: $h(t+s)\ge h(t)$ for $0\le s\le s_0$, where $s_0>0$ depends on $t$.
\item[(iii)] If $h$ became negative, consider the last time $t_0$ before that at which $h(t_0)\ge0$.
The local claim at the single point $t_0$ gives a contradiction. That $s_0$ depends on $t$ does not
matter, since the local claim is used at only one point.
\end{enumerate}
Positivity of $A$ enters through the hypothesis of Lemma~\ref{lem:tangent}, and it does so in a
self-sustaining way. As long as $h(t)\ge 0$, a best matrix $B$ at time $t$ satisfies $B\ge A>0$, so it is
positive and Lemma~\ref{lem:tangent} applies. Comparing with $A$, rather than with the zero matrix,
is what keeps the argument self-sustaining, and it also yields the stronger conclusion $B\ge A$.
The maximum over $Q$ is taken separately for each $t$. This means that we never have to control an actual path
$t\mapsto B(t)$ of matrices, and the proof is not constructive. 

Finally, Step~2 removes
the assumptions that $A>0$ and that $\vd$ lies in the interior of $\Kn$ by a compactness argument.

\subsection*{Step 1: positive $A$ and interior $\vd$}
Assume first that $A>0$ and that $d_1>\sum_{i\ge2}|d_i|$, put $m=\min_{p,q}a_{pq}>0$, and let
$\Lambda$ and $D(t)$ be as in the outline. Define
\[
g(Q,t)=\min_{p,q}\bigl(Q\,D(t)\,Q^{\top}-A\bigr)_{pq},\qquad h(t)=\max_{Q\in\On}g(Q,t).
\]
The maximum is attained because $\On$ is compact and $g$ is continuous. We call $Q\in\On$ a
\emph{maximizer at $t$} if $g(Q,t)=h(t)$. For such $Q$ the matrix $QD(t)Q^{\top}$ has eigenvalues
$\lambda_i+td_i$, $i=1,\dots,n$, and satisfies $QD(t)Q^{\top}\ge A+h(t)J$.

\smallskip\noindent\emph{Claim 1. $h$ is continuous.} We prove $|h(t)-h(t')|\le\|\vd\|_\infty|t-t'|$ for
all $t,t'\ge0$.
Let $Q\in\On$ and $t,t'\ge0$. Since $D(t)-D(t')=(t-t')\diag(\vd)$, we have for all $p,q$
\[
\bigl(QD(t)Q^{\top}-A\bigr)_{pq}-\bigl(QD(t')Q^{\top}-A\bigr)_{pq}=(t-t')\,\bigl(Q\diag(\vd)Q^{\top}\bigr)_{pq},
\]
and, since the rows of $Q$ are unit vectors, the Cauchy--Schwarz inequality gives
\[
\bigl|\bigl(Q\diag(\vd)Q^{\top}\bigr)_{pq}\bigr|=\Bigl|\sum_i d_i\,q_{pi}q_{qi}\Bigr|
\le\|\vd\|_\infty\sum_i|q_{pi}|\,|q_{qi}|\le\|\vd\|_\infty .
\]
Thus every entry of $QD(t)Q^{\top}-A$ changes by at most $\|\vd\|_\infty|t-t'|$ when $t$ is replaced by
$t'$, and hence so does their minimum: $|g(Q,t)-g(Q,t')|\le\|\vd\|_\infty|t-t'|$ for every $Q\in\On$.
If $Q$ is a maximizer at $t$, then
\[
h(t')\ \ge\ g(Q,t')\ \ge\ g(Q,t)-\|\vd\|_\infty|t-t'|\ =\ h(t)-\|\vd\|_\infty|t-t'| ,
\]
and exchanging the roles of $t$ and $t'$ gives $|h(t)-h(t')|\le\|\vd\|_\infty|t-t'|$.

\smallskip

Taking $Q_0$ with $Q_0\Lambda Q_0^{\top}=A$ shows $h(0)\ge 0$. We claim that $h(t)\ge 0$ for all $t\ge 0$. For $t=1$, a maximizer $Q$ at $1$ then gives
$B=QD(1)Q^{\top}\ge A$ with the required spectrum.

\smallskip\noindent\emph{Claim 2. If $h(t)>-m$, then there is $s_0>0$ with
$h(t+s)\ge h(t)$ for all $s\in[0,s_0]$.}

Let $Q$ be a maximizer at $t$ and $B=QD(t)Q^{\top}$. Then $B\ge A+h(t)J>0$, so $B$ is a positive matrix. In particular, $\lambda_1+td_1$ is the Perron eigenvalue of $B$ and  $\vq_1$ (the first column of $Q$) is a corresponding eigenvector that is positive up to sign. Replacing $\vq_1$ by $-\vq_1$ does not change
$QD(t)Q^{\top}$, so we may assume $\vq_1>0$. 
By Lemma~\ref{lem:tangent} there is $K\in\Ln$ with
$X=Q\diag(\vd)Q^{\top}+KB-BK>0$. Let $\xi=\min_{p,q}X_{pq}>0$ and $Q(s)=e^{sK}Q\in\On$. Then, as in \eqref{eq:curve} and
\eqref{eq:expansion} with $\Lambda_B=D(t)$,
\[
Q(s)D(t+s)Q(s)^{\top}=e^{sK}\bigl(B+s\,Q\diag(\vd)Q^{\top}\bigr)e^{-sK}=B+sX+R(s),
\]
where every entry of $R(s)$ is bounded by $Cs^2$ for $s\in[0,1]$, with $C$ depending on $B$, $K$ and $\vd$.
Hence
\[
h(t+s)\ \ge\ g(Q(s),t+s)\ \ge\ \min_{p,q}(B-A)_{pq}+s\,\xi-Cs^2\ =\ h(t)+s(\xi-Cs),
\]
which is at least $h(t)$ for $0\le s\le s_0=\min\{1,\xi/C\}$.

\smallskip\noindent\emph{Conclusion of Step 1.} Suppose $h(t_1)<0$ for some $t_1>0$, and let
\[
t_0=\sup\{t\in[0,t_1]:\ h(t)\ge 0\}.
\]
Since $h(0)\ge0$, the set in this definition is nonempty, and by continuity of $h$ we have
 $h(t_0)=0$. Since $0>-m$, Claim~2 gives $h(t_0+s)\ge h(t_0)=0$ for all small $s>0$,
a contradiction.

\subsection*{Step 2: general $A\in\Snp$ and $\vd\in\Kn$}
Let $A\in\Snp$, $\vd\in\Kn$, $A_k=A+\tfrac1k J>0$, with eigenvalues
$\lambda^{(k)}_1\ge\dots\ge\lambda^{(k)}_n$, and $\vd^{(k)}=\vd+\tfrac1k\ve_1$, which lies in the
interior of $\Kn$. By Step~1 there are $B_k\in\Sn$ with $B_k\ge A_k\ge A$ and
eigenvalues $\lambda^{(k)}_i+d^{(k)}_i$, $i=1,\dots,n$. The ordered eigenvalues of a symmetric matrix depend continuously on the
matrix, 
so $\lambda^{(k)}_i\to\lambda_i$, and the Frobenius norms
$\|B_k\|_F^2=\sum_i(\lambda^{(k)}_i+d^{(k)}_i)^2$ are bounded. A convergent subsequence has a limit $B\in\Snp$
with $B\ge A$ and, by continuity of the spectrum, eigenvalues $\lambda_i+d_i$. \qed

\section{Guo-extreme lists}\label{sec:extreme}

Theorem~\ref{thm:main} says that the set $\Rn$ of symmetrically realizable lists is invariant under
adding vectors from the cone $\Kn$. This suggests singling out the realizable lists at which no such
move is reversible. In this section we assume $n\ge 2$. 

\begin{definition}\label{def:extreme}
Let $\sigma\in\Rn$. For $\vv\in\{\ve_1\pm\ve_i:\ i=2,\dots,n\}$ we say that the move in direction $\vv$
is \emph{reversible at $\sigma$} if $\sigma-t\vv\in\Rn$ for some $t>0$. The list $\sigma$ is
\emph{Guo-extreme} if no such move is reversible at $\sigma$, that is, if for every
$i\in\{2,\dots,n\}$ and every $t>0$, neither $\sigma-t(\ve_1+\ve_i)$ nor $\sigma-t(\ve_1-\ve_i)$
belongs to $\Rn$. The set of Guo-extreme lists in $\Rn$ is denoted by $\En$.
\end{definition}

The condition on the Perron eigenvalue, which is built into the definition of $\Rn$, excludes the
trivial case $t=\lambda_1-\lambda_i$ of the second move, in which the two eigenvalues are merely
interchanged. By part~(a) of the next proposition only small $t$ matter.

\begin{proposition}\label{prop:rays}
Let $\sigma\in\Rn$.
\begin{enumerate}
\item[(a)] For each $\vv\in\{\ve_1\pm \ve_i:\ i=2,\dots,n\}$ the set $\{t\ge 0:\ \sigma-t\vv\in\Rn\}$ is a
compact interval $[0,\tau_\vv]$. Hence $\sigma$ is Guo-extreme if and only if $\tau_\vv=0$ for all $\vv\in\{\ve_1\pm \ve_i:\ i=2,\dots,n\} $.
\item[(b)] $\sigma$ is Guo-extreme if and only if $\sigma-\vd\notin\Rn$ for every $\vd\in\Kn\setminus\{0\}$,
that is, if and only if $\sigma$ is minimal in $\Rn$ for the partial order defined by the cone $\Kn$.
\end{enumerate}
\end{proposition}

\begin{proof}
(a) If $\sigma-t\vv\in\Rn$ and $0\le s\le t$, then $\sigma-s\vv=(\sigma-t\vv)+(t-s)\vv\in\Rn+\Kn\subseteq\Rn$
by \eqref{eq:RnKn}. So the set $\{t\ge 0:\ \sigma-t\vv\in\Rn\}$ is an interval containing $0$. It is closed because $\Rn$ is
closed, and bounded because the first coordinate $\lambda_1-t$ of $\sigma-t\vv$ must be nonnegative.

(b) Suppose $\sigma-\vd\in\Rn$ for some $\vd\in\Kn\setminus\{0\}$. By \eqref{eq:Kngen} we can write
$\vd=c\vv+\vd'$ with $c>0$, $\vv=\ve_1\pm \ve_i$ for some $i$, and $\vd'\in\Kn$. Then
$\sigma-c\vv=(\sigma-\vd)+\vd'\in\Rn$, so $\sigma$ is not Guo-extreme. The converse is clear, since
$t(\ve_1\pm \ve_i)\in\Kn\setminus\{0\}$ for $t>0$.
\end{proof}

By part~(a), the move in direction $\vv$ is reversible at $\sigma$ if and only if $\sigma-t\vv\in\Rn$
for all sufficiently small $t>0$.  

\begin{proposition}\label{prop:decomposition}
$\Rn=\En+\Kn$. That is, every symmetrically realizable list is obtained from a Guo-extreme list by
adding a vector from the cone $\Kn$.
\end{proposition}

\begin{proof}
The inclusion $\En+\Kn\subseteq\Rn$ is \eqref{eq:RnKn}. Conversely, let $\sigma\in\Rn$ and
\[
\mathcal{T}=\{\tau\in\Rn:\ \sigma-\tau\in\Kn\}.
\]
This set contains $\sigma$ and is closed. It is bounded: if $\tau=\sigma-\vd\in\mathcal{T}$, then
$\tau_1=\lambda_1-d_1\ge 0$ by \eqref{eq:nec}, so $|d_i|\le d_1\le\lambda_1$ for all $i$. Let
$\sigma^*\in\mathcal{T}$ be a point with the smallest first coordinate. If $\sigma^*-\vd'\in\Rn$ for some
$\vd'\in\Kn\setminus\{0\}$, then $\sigma-(\sigma^*-\vd')=(\sigma-\sigma^*)+\vd'\in\Kn$, so
$\sigma^*-\vd'\in\mathcal{T}$, and its first coordinate is smaller than that of $\sigma^*$ because
$d'_1>0$. This is impossible, so $\sigma^*\in\En$ by Proposition~\ref{prop:rays}(b), and
$\sigma=\sigma^*+(\sigma-\sigma^*)\in\En+\Kn$.
\end{proof}

Thus the symmetric nonnegative inverse eigenvalue problem is equivalent to the description of the
Guo-extreme lists. The following simple result offers a family of such lists.

\begin{proposition}\label{prop:criterion}
Let $\sigma=(\lambda_1,\dots,\lambda_n)\in\Rn$.
\begin{enumerate}
\item[(a)] If $\sum_{i=1}^n\lambda_i=0$, then $\sigma-t(\ve_1+\ve_i)\notin\Rn$ for every
$i\in\{2,\dots,n\}$ and every $t>0$.
\item[(b)] If $\sum_{i=1}^n\lambda_i^3=0$ and $i\in\{2,\dots,n\}$ satisfies $\lambda_1+\lambda_i>0$,
then $\sigma-t(\ve_1-\ve_i)\notin\Rn$ for every $t>0$.
\end{enumerate}
In particular, if
\[
\sum_{i=1}^n\lambda_i=0,\qquad \sum_{i=1}^n\lambda_i^3=0,\qquad \lambda_1+\lambda_i>0\ \text{ for }
i=2,\dots,n,
\]
then $\sigma$ is Guo-extreme.
\end{proposition}

\begin{proof}
(a) The sum of the entries of $\sigma-t(\ve_1+\ve_i)$ is $-2t<0$, so this list is not in $\Rn$ by
\eqref{eq:nec}.

(b) Let $t>0$ and $\sigma'=\sigma-t(\ve_1-\ve_i)$. If $t>(\lambda_1-\lambda_i)/2$, then the first entry
$\lambda_1-t$ of $\sigma'$ is smaller than its $i$-th entry $\lambda_i+t$, so $\sigma'\notin\Rn$. If
$0<t\le(\lambda_1-\lambda_i)/2$, then $\lambda_1-\lambda_i-t>0$, and the sum of the cubes of the
entries of $\sigma'$ is
\[
(\lambda_1-t)^3+(\lambda_i+t)^3-\lambda_1^3-\lambda_i^3
=-3t\,(\lambda_1+\lambda_i)(\lambda_1-\lambda_i-t)<0 ,
\]
so again $\sigma'\notin\Rn$ by \eqref{eq:nec}. The last statement follows from (a) and (b).
\end{proof}

\begin{example}\label{ex:trianglefree}
Let $A\in\Snp$ be irreducible with zero diagonal, and suppose that the graph of $A$ has no triangles and
is not bipartite. Then $\tr A=0$ and $\tr A^3=0$, and $-\rho(A)$ is not an eigenvalue of $A$ because
the graph is not bipartite \cite[Chapter~2]{BP}. By Proposition~\ref{prop:criterion} the spectrum of
$A$, with $\rho(A)$ listed first, is Guo-extreme. In particular the spectrum
$\bigl(2,\gamma,\gamma,-1-\gamma,-1-\gamma\bigr)$, $\gamma=\tfrac{\sqrt5-1}{2}$, of the $5$-cycle
\cite[Example~14]{LS} is Guo-extreme, and so are the spectra of all odd cycles of length at least $5$
and of the Petersen graph. The hypothesis $\lambda_1+\lambda_i>0$ in
Proposition~\ref{prop:criterion} cannot be dropped: the list $(1,-1)$ has zero sum and zero sum of
cubes, but $(1-t,-1+t)\in\mathcal{R}_2$ for $0\le t\le 1$.
\end{example}

\section{Positivity patterns and a certificate}\label{sec:certificate}

In this section we give sufficient conditions for a list not to be Guo-extreme. We first
characterize the spectra of positive symmetric matrices and of symmetric nonnegative matrices with
positive off-diagonal entries in terms of the cone $\Kn$ (Propositions~\ref{prop:perron}
and~\ref{prop:offdiag}). We then give a certificate on a realizing matrix $A$
(Theorem~\ref{thm:certificate}) which guarantees that $A$ is orthogonally similar to a positive matrix or to a matrix with positive off-diagonal elements. In particular,  the spectrum of $A$ is not Guo-extreme. As in
Section~\ref{sec:extreme} we assume $n\ge 2$.

\subsection*{Spectra of positive matrices and of matrices with positive off-diagonal entries}
Recall from \cite{L98} that a symmetrically realizable list $\sigma$ is an \emph{extreme symmetric
spectrum} if $\sigma-t(1,\dots,1)$ is not symmetrically realizable for any $t>0$, and that a matrix in
$\Snp$ with such a spectrum is an \emph{extreme symmetric matrix}. Recall also from \cite{LS} that a
list $\sigma\in\Rn$ is \emph{Perron extreme} if $\sigma-t\ve_1\notin\Rn$ for all $t>0$.

The following consequence of Theorem~\ref{thm:main} is a technical result used twice below. Here $\|\cdot\|_F$ is the
Frobenius norm.

\begin{lemma}\label{lem:perturb}
Let $\sigma'=(\lambda'_1,\dots,\lambda'_n)\in\Rn$, let $A'\in\Snp$ realize $\sigma'$, let $E\in\Snp$ and $\vd\in\R^n$, and suppose that one of the following
holds:
\begin{enumerate}
\item[(i)] $d_1-\sum_{i=2}^n|d_i|\ge n\|E\|_F$;
\item[(ii)] $\tr E=0$, $\vd\in\Hc$, and $d_i\le-\|E\|_F$ for $i=2,\dots,n$.
\end{enumerate}
Then there is $B\in\Sn$ with $B\ge A'+E$ and eigenvalues $\lambda'_1+d_1,\dots,\lambda'_n+d_n$.
\end{lemma}

\begin{proof}
Conditions (i) and (ii) and the conclusion are unchanged if the coordinates $2,\dots,n$ of $\sigma'$
and $\vd$ are permuted simultaneously, so we may assume that $\lambda'_2\ge\dots\ge\lambda'_n$. Let
$\mu_1\ge\dots\ge\mu_n$ be the eigenvalues of $A'+E$, let $\eta=(\mu_1,\dots,\mu_n)-\sigma'$, and
put $\varepsilon=\|E\|_F$. Since $\lambda'_1\ge\lambda'_2\ge\dots\ge\lambda'_n$ are the eigenvalues of
$A'$ in decreasing order, the theorem of Hoffman and Wielandt \cite[Section~6.3]{HJ} gives
$\sum_i\eta_i^2\le\|(A'+E)-A'\|_F^2=\varepsilon^2$, so $|\eta_i|\le\varepsilon$ for all $i$. Moreover,
$\sum_i\eta_i=\tr(A'+E)-\tr A'=\tr E$.

We show that $\vd-\eta\in\Kn$. In case (i),
\[
(\vd-\eta)_1-\sum_{i\ge 2}|(\vd-\eta)_i|\ \ge\ d_1-\varepsilon-\sum_{i\ge 2}\bigl(|d_i|+\varepsilon\bigr)
\ =\ d_1-\sum_{i\ge 2}|d_i|-n\varepsilon\ \ge\ 0 .
\]
In case (ii), $(\vd-\eta)_i\le d_i+\varepsilon\le 0$ for $i\ge 2$, and $\sum_i(\vd-\eta)_i=\sum_id_i-\tr
E=0$, so $(\vd-\eta)_1=\sum_{i\ge 2}|(\vd-\eta)_i|$.

Theorem~\ref{thm:main}, applied to $A'+E$ and $\vd-\eta$, now gives $B\in\Sn$ with $B\ge A'+E\ge 0$
and eigenvalues $\mu_i+d_i-\eta_i=\lambda'_i+d_i$. Since $\mu_1\ge\mu_i$ and $\vd-\eta\in\Kn$, the
largest of them is $\lambda'_1+d_1=\rho(B)$, so $\sigma'+\vd\in\Rn$.
\end{proof}

In words, adding $E$ to $A'$ creates positive entries but perturbs the spectrum slightly.
Theorem~\ref{thm:main} then moves the spectrum to $\sigma'+\vd$ without decreasing any entry, which
is possible as long as $\vd$ lies far enough inside $\Kn$ compared with the size of $E$.

The spectra of positive symmetric matrices are precisely the lists that are not Perron extreme.

\begin{proposition}\label{prop:perron}
For $\sigma\in\Rn$ the following statements are equivalent.
\begin{enumerate}
\item[(a)] $\sigma$ is not Perron extreme.
\item[(b)] $\sigma-\vd\in\Rn$ for some $\vd$ in the interior of $\Kn$, that is, for some $\vd$ with
$d_1>\sum_{i=2}^n|d_i|$.
\item[(c)] $\sigma$ is the spectrum of a positive symmetric matrix.
\end{enumerate}
\end{proposition}

\begin{proof}
(a)$\Rightarrow$(b): If $\sigma-t\ve_1\in\Rn$ with $t>0$, take $\vd=t\ve_1$.

(b)$\Rightarrow$(c): Let $\sigma'=\sigma-\vd\in\Rn$ with $\vd$ in the interior of $\Kn$. Let $A'\in\Snp$ realize $\sigma'$, and apply
Lemma~\ref{lem:perturb}(i) with $E=\delta J$, where $\delta>0$ is so small that
$n\|\delta J\|_F=n^2\delta\le d_1-\sum_{i\ge 2}|d_i|$. It gives $B\in\Sn$ with $B\ge A'+\delta J>0$
and eigenvalues $\lambda'_i+d_i$, which are the entries of $\sigma$.

(c)$\Rightarrow$(a): Let $B$ be a positive symmetric matrix with spectrum $\sigma$ and unit Perron
vector $\vu>0$. For small $t>0$ the matrix $B-t\vu\vu^{\top}$ is nonnegative, its eigenvalues are
$\lambda_1-t,\lambda_2,\dots,\lambda_n$, and $\lambda_1-t$ is the largest of them because $\lambda_1$ is
a simple eigenvalue of $B$. Hence $\sigma-t\ve_1\in\Rn$.
\end{proof}

Proposition~\ref{prop:perron} describes the lists obtained by adding to an element of $\Rn$ a vector
from the interior of $\Kn$: they are the spectra of positive matrices. There is an analogous
description for the face of $\Kn$ that consists of the trace-preserving directions,
\[
\Kn\cap\Hc
=\Bigl\{\vd\in\R^n:\ d_i\le 0\ \text{for } i\ge 2,\ d_1=-\sum_{i=2}^nd_i\Bigr\},
\]
which is generated by the vectors $\ve_1-\ve_i$, $i=2,\dots,n$.

\begin{proposition}\label{prop:offdiag}
Let $\sigma=(\lambda_1,\dots,\lambda_n)\in\Rn$. The following statements are equivalent.
\begin{enumerate}
\item[(a)] $\sigma$ is the spectrum of a matrix in $\Snp$ all of whose off-diagonal entries are
positive.
\item[(b)] There are $t_2,\dots,t_n>0$ such that
$\bigl(\lambda_1-\sum_{i=2}^nt_i,\ \lambda_2+t_2,\ \dots,\ \lambda_n+t_n\bigr)\in\Rn$.
\end{enumerate}
\end{proposition}

\begin{proof}
(a)$\Rightarrow$(b): Let $B\in\Snp$ have spectrum $\sigma$ and positive off-diagonal entries. By
Corollary~\ref{cor:fullsupport} below, applied to $B$ and $\vd=(-(n-1),1,\dots,1)$, we have
$\sigma+s\vd\in\Rn$ for small $s>0$, which is (b) with $t_2=\dots=t_n=s$.

(b)$\Rightarrow$(a): Let $\vd=\bigl(\sum_{i\ge 2}t_i,-t_2,\dots,-t_n\bigr)$, so that
$\sigma'=\sigma-\vd\in\Rn$. Let $A'\in\Snp$ realize
$\sigma'$, and apply Lemma~\ref{lem:perturb}(ii) with $E=\delta(J-I)$, where $\delta>0$ is so small
that $\|\delta(J-I)\|_F\le\min_{i\ge 2}t_i$. Note that $\tr E=0$, $\vd\in\Hc$ and $d_i=-t_i$ for
$i\ge 2$. It gives $B\in\Sn$ with $B\ge A'+\delta(J-I)$ and eigenvalues
$\lambda'_i+d_i$, which are the entries of $\sigma$. The off-diagonal entries of $B$ are at least
$\delta$.
\end{proof}

Condition (b) in Proposition~\ref{prop:offdiag} says that the trace-preserving moves of
Corollary~\ref{cor:guo} are \emph{simultaneously reversible at $\sigma$}: there are $t_2,\dots,t_n>0$
with $\sigma-\sum_{i\ge 2}t_i(\ve_1-\ve_i)\in\Rn$. Not being Guo-extreme only requires that one move is
reversible at $\sigma$, so by Proposition~\ref{prop:offdiag} the spectrum of a matrix
in $\Snp$ with positive off-diagonal entries is not Guo-extreme. The next example shows that the
converse fails.

\begin{example}\label{ex:path}
The list $\sigma=(1,0,-1)$ is not Guo-extreme: for $0\le t<1$ and $a,b>0$ with $a^2+b^2=(1-t)^2$ the
matrix
\[
\begin{pmatrix}0&a&0\\a&0&b\\0&b&0\end{pmatrix}
\]
has spectrum $(1-t,0,-1+t)=\sigma-t(\ve_1-\ve_3)$. On the other hand, every matrix $A\in\mathcal{S}_3^+$ with
spectrum $\sigma$ has zero diagonal, because $\tr A=0$, and then $0=\det A=2a_{12}a_{13}a_{23}$, so some
off-diagonal entry of $A$ vanishes. Thus $\sigma$ does not satisfy the conditions of
Proposition~\ref{prop:offdiag}.
\end{example}

\begin{remark}\label{rem:cone}
Propositions~\ref{prop:perron} and~\ref{prop:offdiag} relate parts of the cone $\Kn$ to positivity
patterns of realizing matrices. A list $\sigma\in\Rn$ can be written as $\sigma'+\vd$ with
$\sigma'\in\Rn$ and $\vd$ in the interior of $\Kn$ if and only if it is the spectrum of a positive
matrix, and with $\vd$ in the relative interior of the trace-preserving face of $\Kn$ if and only if
it is the spectrum of a matrix in $\Snp$ with positive off-diagonal entries. Guo-extremality concerns
the extreme rays $\ve_1\pm\ve_i$ of $\Kn$, and Example~\ref{ex:path} shows that it is not determined by
a positivity pattern in the same way.
\end{remark}

\subsection*{A certificate for positivity up to orthogonal similarity}
We now give a condition on a realizing matrix $A$ which guarantees that $A$ is orthogonally similar
to a positive matrix or to a matrix with zero diagonal and positive off-diagonal entries; in
particular, its spectrum is not Guo-extreme. It is modeled on a theorem of Laffey. Laffey proved \cite[Theorem~4]{L98} that for every extreme symmetric
matrix $A$ there is a nonzero nonnegative symmetric matrix $Y$ with $AY=YA$ and $A\circ Y=0$. Equivalently: if no such $Y$ exists, then the spectrum of $A$ is not extreme. For
$A\in\Snp$ let
\[
\YA=\{Y\in\Snp:\ AY=YA,\ A\circ Y=0\}.
\]
Thus $Y\in\YA$ vanishes in every position where $A$ is positive. Since
$\tr(AY)=\sum_{p,q}a_{pq}y_{pq}$ and, for $Y\in\Snp$, all terms are nonnegative, the condition
$A\circ Y=0$ is equivalent to $\tr(AY)=\ip{A}{Y}=0$. The set $\YA$ is a closed convex cone, and deciding whether it
contains a matrix that is not a multiple of $I$ is a linear feasibility problem. In this notation
Laffey's theorem says that $\YA\ne\{0\}$ for every extreme symmetric matrix $A$.

In fact the condition $\YA=\{0\}$ certifies more than non-extremality. As observed by Loewy and
Spector \cite[Theorem~1]{LSp}, the only property of an extreme symmetric matrix used in Laffey's proof
is that it is not orthogonally similar to a positive symmetric matrix. Thus, if $\YA=\{0\}$, then $A$
is orthogonally similar to a positive symmetric matrix, and by Proposition~\ref{prop:perron} its
spectrum is not Perron extreme. Consequently $\YA\ne\{0\}$ for every $A\in\Snp$ whose spectrum is Perron extreme. 

Note that $I\in\YA$ if and only if $A$ has zero diagonal; thus the certificate $\YA=\{0\}$ is never
available for matrices with zero trace, whose spectra are always extreme and Perron extreme.
Theorem~\ref{thm:certificate} below allows $\YA$ to contain the multiples of $I$. Part~(a) of this
theorem is the observation of Loewy and Spector recalled above.  Part~(b) applies to matrices with trace-zero
and is new. Its consequences for the spectrum are collected in Corollary~\ref{cor:moves}. Both
results use the correction step of Lemma~\ref{lem:correction}.

\begin{theorem}\label{thm:certificate}
Let $A\in\Snp$, and suppose that every matrix in $\YA$ is a scalar multiple of $I$.
\begin{enumerate}
\item[(a)] If $A$ has a positive diagonal entry (so that the hypothesis says $\YA=\{0\}$), then $A$
is orthogonally similar to a positive matrix.
\item[(b)] If $A$ has zero diagonal, then $A$ is irreducible and orthogonally similar to a matrix with
zero diagonal and positive off-diagonal entries.
\end{enumerate}
\end{theorem}

\begin{proof}
Let $\Zc=\{(p,q)\in\Pi:\ a_{pq}=0\}$ be the set of zero positions of $A$, including diagonal ones, and
recall that $\Delta$ is the set of diagonal positions. In case (b), $\Delta\subseteq\Zc$.

\smallskip\noindent\emph{Step 1: a first-order direction.} We claim that there is $K\in\Ln$ such that
$X=KA-AK$ satisfies
\begin{equation}\label{eq:strict}
\begin{cases}
x_{pq}>0\ \text{ for }(p,q)\in\Zc, &\text{in case (a)},\\[2pt]
x_{pq}>0\ \text{ for }(p,q)\in\Zc\setminus\Delta\ \text{ and }\ \dg(X)=0, &\text{in case (b)}.
\end{cases}
\end{equation}
We apply Lemma~\ref{lem:alternative}(b) with $B=A$, $M=0$ and the following partition of the set
$\Pi$ of positions. In case (a), $\mathcal{P}=\Zc$ and $\mathcal{E}=\emptyset$; in case (b),
$\mathcal{P}=\Zc\setminus\Delta$ and $\mathcal{E}=\Delta$. In both cases $\mathcal{F}=\Pi\setminus\Zc$ is
the set of positions where $A$ is positive, and \eqref{eq:strict} is equivalent to $x_{pq}>0$ for
$(p,q)\in\mathcal{P}$ and $x_{pq}=0$ for $(p,q)\in\mathcal{E}$. If no such $K$ exists, then, since
$M=0$, Lemma~\ref{lem:alternative}(b) gives $Y\in\Sn$ with $AY=YA$, $y_{pq}\ge 0$ for
$(p,q)\in\mathcal{P}$, $y_{pq}=0$ for $(p,q)\in\mathcal{F}$, and $y_{pq}\ne0$ for some
$(p,q)\in\mathcal{P}$.

In case (a), every position belongs to $\mathcal{P}$ or $\mathcal{F}$, so $Y$ is nonnegative and
vanishes wherever $A$ is positive. That is, $Y\in\YA$. Since $A$ has a positive diagonal entry,
$I\notin\YA$, so the hypothesis gives $\YA=\{0\}$ and $Y=0$, a contradiction.

In case (b), choose $c\ge\max_p|y_{pp}|$. Then $Y+cI$ is nonnegative, commutes with $A$, and vanishes
wherever $A$ is positive, so $Y+cI\in\YA$. By hypothesis $Y+cI$ is a multiple of $I$, hence $Y$ is a
multiple of $I$ and vanishes on $\mathcal{P}=\Zc\setminus\Delta$, again a contradiction. This proves
the claim.

\smallskip\noindent\emph{Step 2: case (a).} Let $C(s)=e^{sK}Ae^{-sK}$ with $K$ from Step~1. Then
$C(s)$ is orthogonally similar to $A$, and $C(s)=A+sX+O(s^2)$. An entry of $C(s)$ in a position
where $A$ is positive remains positive for small $s$. In a position where $A$ vanishes, the entry of
$C(s)$ equals $s\,x_{pq}+O(s^2)$ with $x_{pq}>0$ by \eqref{eq:strict}, so it is positive for small
$s>0$. Hence $C(s)$ is positive for small $s>0$.

\smallskip\noindent\emph{Step 3: case (b).} We first note that $A$ is irreducible. Otherwise, after a
permutation of the indices, $A=A_1\oplus A_2$ with square blocks $A_1$ and $A_2$. Let $Y$
be the matrix equal to the identity on the indices of $A_1$ and to $0$ elsewhere. Now, $Y$ is nonnegative, commutes with
$A$, and satisfies $A\circ Y=0$  ($A$ has zero diagonal), hence  $Y\in\YA$. This contradicts the hypothesis, since $Y$ is not a multiple of $I$. 

Let $K\in\Ln$ and
$X=KA-AK$ be the matrices obtained in Step~1. In particular, the matrix $X$ has zero diagonal. By Lemma~\ref{lem:correction}(ii), applied with
$\vd=0$, $M=0$ and  $K$, there is a curve $B(s)=e^{L(s)}Ae^{-L(s)}$ of matrices orthogonally
similar to $A$, with zero diagonal and $B(s)=A+sX+O(s^2)$. An off-diagonal entry of $B(s)$ in a
position where $A$ is positive remains positive for small $s$, and an off-diagonal entry in a
position $(p,q)\in\Zc\setminus\Delta$ equals $s\,x_{pq}+O(s^2)$ with $x_{pq}>0$ by \eqref{eq:strict}.
Hence, for small $s>0$, $B(s)$ has zero diagonal and positive off-diagonal entries.
\end{proof}

\begin{corollary}\label{cor:moves}
Let $A\in\Snp$, assume every matrix in $\YA$ is a scalar multiple of $I$, and let
$\sigma=(\lambda_1,\dots,\lambda_n)$ be the spectrum of $A$ with $\lambda_1=\rho(A)$.
\begin{enumerate}
\item[(a)] If $A$ has a positive diagonal entry, then for every $\vd\in\R^n$ we have
$\sigma+s\vd\in\Rn$ for all sufficiently small $s\ge0$.
\item[(b)] If $A$ has zero diagonal, then for every $\vd\in\Hc$ we have
$\sigma+s\vd\in\Rn$ for all sufficiently small $s\ge0$, realized by matrices with zero diagonal.
\end{enumerate}
In both cases $\sigma-s(\ve_1-\ve_i)\in\Rn$ for every $i\in\{2,\dots,n\}$ and all sufficiently small
$s>0$. In particular, $\sigma$ is not Guo-extreme.
\end{corollary}

\begin{proof}
By Theorem~\ref{thm:certificate}, $\sigma$ is the spectrum of a matrix $A'\in\Snp$ that is positive in
case (a), and has zero diagonal and positive off-diagonal entries in case (b). In both cases $A'$ is
irreducible, so $\lambda_1$ is a simple eigenvalue, $\lambda_1>\lambda_j$ for $j\ge2$, and
$\lambda_1+sd_1$ is the largest of the numbers $\lambda_j+sd_j$ for small $s$. Let $Q\in\On$ with
$Q^{\top}A'Q=\diag(\lambda_1,\dots,\lambda_n)$ and $M=Q\diag(\vd)Q^{\top}$.

(a) The matrix $A'+sM$ has eigenvalues $\lambda_j+sd_j$ and is positive for small $s$, so
$\sigma+s\vd\in\Rn$.

(b) By Lemma~\ref{lem:correction}(i) there is $K\in\Ln$ with $\dg(M+KA'-A'K)=0$, and now
Lemma~\ref{lem:correction}(ii) gives is the existence of a curve $B(s)$ of symmetric matrices with eigenvalues $\lambda_j+sd_j$,
zero diagonal, and $B(s)=A'+O(s)$. Since the off-diagonal entries of $A'$ are positive, so are those
of $B(s)$ for small $s$, and hence $\sigma+s\vd\in\Rn$.

The last statement follows by taking $\vd=\ve_i-\ve_1$.
\end{proof}

\begin{corollary}\label{cor:fullsupport}
Let $A\in\Snp$ have positive off-diagonal entries, $n\ge 2$, and let $\sigma=(\lambda_1,\dots,\lambda_n)$
be its spectrum with $\lambda_1=\rho(A)$. Then for every $\vd\in\Hc$ we have
$\sigma+s\vd\in\Rn$ for all sufficiently small $s>0$. In particular $\sigma$ is not Guo-extreme. If $A$
has zero diagonal, the realizing matrices can be chosen with zero diagonal.
\end{corollary}

\begin{proof}
The matrix $A$ is irreducible. A matrix $Y\in\YA$ vanishes off the diagonal, and a diagonal matrix
that commutes with the irreducible matrix $A$ is a multiple of $I$. If $A$ has a positive diagonal
entry, then a multiple of $I$ in $\YA$ must vanish there, so $\YA=\{0\}$. In both cases
Corollary~\ref{cor:moves} applies.
\end{proof}

Corollary~\ref{cor:fullsupport} only covers matrices without off-diagonal zeros. The next example shows
that Theorem~\ref{thm:certificate}(b) also applies to sparse matrices, and that a single additional
edge can destroy Guo-extremality.

\begin{example}\label{ex:house}
Let $A\in\mathcal{S}_5^+$ be the adjacency matrix of the graph obtained from the $5$-cycle
$1-2-3-4-5-1$ by adding the chord $\{1,3\}$ (the \emph{house} graph). Its characteristic polynomial is
$x(x+2)(x^3-2x^2-2x+2)$, so
\[
\sigma\approx(2.4812,\ 0.6889,\ 0,\ -1.1701,\ -2).
\]
We check that every matrix in $\YA$ is a multiple of $I$. A matrix $Y\in\YA$ vanishes at the edges
of the graph, 
and a direct computation shows that $AY-YA$ is the skew-symmetric matrix
\[
\mbox{\small$\setlength{\arraycolsep}{3pt}
\begin{pmatrix}
0 & y_{22}-y_{11}+y_{25} & y_{33}-y_{11}+y_{35}-y_{14} & y_{24} & y_{55}-y_{11}+y_{25}+y_{35}-y_{14}\\
\ast & 0 & y_{33}-y_{22}-y_{24} & y_{14}-y_{25} & y_{35}-y_{24}\\
\ast & \ast & 0 & y_{44}-y_{33}+y_{14}+y_{24}-y_{35} & y_{25}\\
\ast & \ast & \ast & 0 & y_{55}-y_{44}+y_{35}-y_{14}\\
\ast & \ast & \ast & \ast & 0
\end{pmatrix}$},
\]
where the entries marked $\ast$ are determined by skew-symmetry. Setting the entries in the positions
$(1,4)$, $(2,5)$, $(3,5)$ and $(2,4)$ equal to zero, in this order, gives
\[
y_{24}=0,\qquad y_{35}=y_{24}=0,\qquad y_{25}=0,\qquad y_{14}=y_{25}=0 .
\]
Hence $Y$ is diagonal, and a diagonal matrix that commutes with the irreducible matrix $A$ is a
multiple of $I$. (Nonnegativity of $Y$ is not needed here.) By Theorem~\ref{thm:certificate}(b), $\sigma$ is the spectrum of a matrix in $\mathcal{S}_5^+$ with zero
diagonal and positive off-diagonal entries. Hence, by Proposition~\ref{prop:offdiag}, all
trace-preserving moves are simultaneously reversible at $\sigma$; in particular $\sigma$ is
not Guo-extreme.
The same conclusion holds for every matrix $A\in\mathcal{S}_5^+$ with the same zero pattern: for
general positive weights on the edges, the four entries used above become $a_{12}y_{24}$,
$a_{23}y_{35}-a_{45}y_{24}$, $a_{23}y_{25}$ and $a_{12}y_{14}-a_{45}y_{25}$, and the argument is
unchanged.
\end{example}

To conclude, we compare the notions of extremality that appear in this paper. For $n\ge 2$ consider the following
subsets of $\Rn$:
\begin{itemize}
\item $\mathcal{X}_n=\{\sigma\in\Rn:\ \sigma-t\ve_1\notin\Rn\ \text{for all } t>0\}$ is the set of
Perron extreme lists in the sense of \cite{LS}. By Proposition~\ref{prop:perron}, $\mathcal{X}_n$
consists of the lists in $\Rn$ that are not the spectrum of a positive symmetric matrix.
\item $\mathcal{X}^I_n=\{\sigma\in\Rn:\ \sigma-t(1,\dots,1)\notin\Rn\ \text{for all } t>0\}$ is the set
of extreme symmetric spectra in the sense of \cite{L98}. Since $\sigma-t(1,\dots,1)$ is the spectrum
of $B-tI$ when $\sigma$ is the spectrum of $B$, $\mathcal{X}^I_n$ consists of the lists in $\Rn$ that
are not the spectrum of a matrix in $\Snp$ with positive diagonal entries.
\item $\mathcal{X}^{J-I}_n=\bigl\{\sigma\in\Rn:\ \sigma-\sum_{i=2}^nt_i(\ve_1-\ve_i)\notin\Rn\ \text{for
all } t_2,\dots,t_n>0\bigr\}$ is the set of lists at which the trace-preserving moves are not
simultaneously reversible. By Proposition~\ref{prop:offdiag}, $\mathcal{X}^{J-I}_n$ consists of the
lists in $\Rn$ that are not the spectrum of a matrix in $\Snp$ with positive off-diagonal entries.
\item $\En=\{\sigma\in\Rn:\ \sigma-t(\ve_1\pm\ve_i)\notin\Rn\ \text{for all } i\ge 2 \text{ and } t>0\}$
is the set of Guo-extreme lists (Definition~\ref{def:extreme}).
\end{itemize}

These sets are related by the inclusions
\[
\En\subseteq\mathcal{X}^{J-I}_n\subseteq\mathcal{X}_n\qquad\text{and}\qquad
\mathcal{X}^I_n\subseteq\mathcal{X}_n ,
\]
valid for all $n\ge 2$. Indeed, a positive matrix has positive diagonal and off-diagonal entries,
which gives $\mathcal{X}^I_n\subseteq\mathcal{X}_n$ and $\mathcal{X}^{J-I}_n\subseteq\mathcal{X}_n$, and
$\En\subseteq\mathcal{X}^{J-I}_n$ is the last statement of Proposition~\ref{prop:offdiag}. For $n\ge 3$
all three inclusions are strict: the list $(1,0,-1,0,\dots,0)$ lies in $\mathcal{X}^{J-I}_n$ but is not
Guo-extreme (compare with Example~\ref{ex:path}), the spectrum $(n-1,-1,\dots,-1)$ of $J-I$ lies in
$\mathcal{X}_n$ but not in $\mathcal{X}^{J-I}_n$, and the spectrum $(1,\dots,1)$ of $I$ lies in
$\mathcal{X}_n$ but not in $\mathcal{X}^I_n$.

\section*{Declaration of generative AI in the writing process}

Do you dare ask AI your favourite open mathematical question? If you do not, somebody else might. This paper is the result of such an experiment. We asked whether the symmetric version of Guo's result, posed in 2007, holds, and AI produced a proof. The results presented in this note were hence obtained with the assistance of the AI assistant Claude Opus 5.5 (Anthropic). The present manuscript differs considerably from the original proof proposed by AI: the proofs have been simplified, the arguments streamlined, and unnecessary derivations removed.  The steps were organised into a coherent structure. The author has also added material placing the result in the context of extremality in the SNIEP and sought to provide intuition behind the arguments. The author takes full responsibility for the content of the article.

How new are the ideas in the paper? Several elements of the proofs build upon existing litearture. The first-order step is a duality argument already seen in Laffey’s work on extreme matrices, while the local correction, which relies on the implicit function theorem, is a transversality argument of the kind underlying the strong spectral property in the literature on the inverse eigenvalue problem for graphs (here adapted to a moving spectrum and to sign constraints), see \cite{BBFHHLSY}. What appears to be genuinely new, to the author’s knowledge, is the global argument. Instead of constructing a path of matrices, the proof of Theorem 1 follows a single value function ($h(t)$), defined as a maximum taken over the compact set of matrices with the prescribed spectrum. For directions in the interior of the cone of Guo’s moves, the first-order term can be made strictly positive, which prevents this function from decreasing. Directions on the boundary of the cone are then obtained as limits of interior ones, by compactness.

\end{document}